\documentclass[a4paper]{article}

\usepackage{graphicx}
\usepackage{amsmath}
\usepackage{amssymb}
\usepackage{amsthm}
\usepackage{booktabs} 
\usepackage{youngtab} 
\usepackage{paralist}
\usepackage[vlined,linesnumbered,ruled]{algorithm2e}
\usepackage{todonotes}

\usepackage{tikz}
\usetikzlibrary{patterns}
\usepackage{xargs}

\newcommand{\shadetheboxes}[1]{
	\foreach \x/\y in {#1}
      	\fill[pattern color = black!75, pattern=north east lines] (\x,\y) rectangle +(1,1);
	}

\newcommand{\shadetherects}[1]{
	\foreach \x/\y/\xt/\yt in {#1}
      	\fill[pattern color = black!75, pattern=north east lines] (\x,\y) rectangle (\xt,\yt);
	}

\newcommand{\drawthegrid}[1]{
	\draw (0.01,0.01) grid (#1+0.99,#1+0.99);
	}
	
\newcommand{\drawverticallines}[3]{
	\foreach \x in {#2}
		\draw[line width=#3] (\x+0.01,0.01) -- (\x+0.01,#1+0.99);
	}

\newcommand{\drawhorizontallines}[3]{
	\foreach \y in {#2}
		\draw[line width=#3] (0.01,\y+0.01) -- (#1+0.99,\y+0.01);
	}
	
\newcommand{\drawtheclpattern}[1]{
	\foreach \x/\y in {#1}
      	\filldraw (\x,\y) circle (6pt);
	}
	
\newcommand{\drawtheclpatternwhite}[1]{
	\foreach \x/\y in {#1}
      	\draw[fill=white] (\x,\y) circle (6pt);
	}
	
\newcommand{\drawtheclpatternwhitebig}[1]{
	\foreach \x/\y in {#1}
		\draw[fill=white] (\x,\y) circle (11pt);
	}
	
\newcommand{\drawspecialbox}[1]{
	\foreach \x/\y/\z/\w/\A in {#1}
		{
       		\fill[color = white!100, opacity=1, rounded corners = 1.5pt] (\x+0.125,\y+0.125) rectangle (\z-0.125,\w-0.125);
       		\draw[color = black, rounded corners = 1.5pt] (\x+0.125,\y+0.125) rectangle (\z-0.125,\w-0.125);
       		\fill[black] (\x/2+\z/2,\y/2+\w/2) node {$\scriptstyle\A$};
       	}
    }

\newcommand{\onetwo}{\mpatternnl{scale=0.2}{2}{1/1,2/2}{}}		

\newcommand{\drawsolidshadedbox}[1]{
	\foreach \x/\y/\z/\w/\A in {#1}
		{
       		\fill[color = gray!50, opacity=1, rounded corners=1.5pt] (\x+0.125,\y+0.125) rectangle (\z-0.125,\w-0.125);
       		\draw[color = black, rounded corners=1.5pt] (\x+0.125,\y+0.125) rectangle (\z-0.125,\w-0.125);
       		\fill[black] (\x/2+\z/2,\y/2+\w/2) node {$\scriptstyle\A$};
       	}
    }

\newcommand{\drawascendingrestrictions}[1]{
	\foreach \x/\y/\z/\w in {#1}
		{
       		\fill[color = white!100, opacity=1, rounded corners=1.5pt] (\x+0.125,\y+0.125) rectangle (\z-0.125,\w-0.125);
       		\fill[pattern color = black!65, pattern=north east lines, rounded corners=1.5pt] (\x+0.125,\y+0.125) rectangle (\z-0.125,\w-0.125);
       		\draw[color = black, rounded corners=1.5pt] (\x+0.125,\y+0.125) rectangle (\z-0.125,\w-0.125);
       		\fill[black] (\x/2+\z/2,\y/2+\w/2) node {\onetwo};
       	}
    }    

\newcommand{\mpattern}[4]{										
  \raisebox{0.6ex}{
  \begin{tikzpicture}[scale=0.35, baseline=(current bounding box.center), #1]
  	\useasboundingbox (0.0,-0.1) rectangle (#2+1.4,#2+1.1);
	
    \shadetheboxes{#4}
    
    \drawthegrid{#2}
    
    \drawtheclpattern{#3}
    
  \end{tikzpicture}}
}

\newcommand{\mpatternnl}[4]{										
  \raisebox{0.6ex}{
  \begin{tikzpicture}[scale=0.35, baseline=(current bounding box.center), #1]
  	\useasboundingbox (0.85,-0.1) rectangle (#2+1.4,#2+1.1);
	
    \shadetheboxes{#4}
    
    \drawtheclpattern{#3}
    
  \end{tikzpicture}}
}

\newcommand{\mpatternwwo}[6]{									
  \raisebox{0.6ex}{
  \begin{tikzpicture}[scale=0.35, baseline=(current bounding box.center), #1]
  \useasboundingbox (0.0,-0.1) rectangle (#2+1.4,#2+1.1);
  
    \shadetheboxes{#6}
    
    \drawthegrid{#2}
    
    \drawtheclpatternwhite{#4}
    \drawtheclpatternwhitebig{#5}
    \drawtheclpattern{#3}
    
  \end{tikzpicture}}
}

\newcommand{\mmpattern}[5]{									
  \raisebox{0.6ex}{
  \begin{tikzpicture}[scale=0.35, baseline=(current bounding box.center), #1]
  \useasboundingbox (0.0,-0.1) rectangle (#2+1.4,#2+1.1);
    
    \shadetheboxes{#4}
    
    \drawthegrid{#2}
    
    \drawspecialbox{#5}
    
    \drawtheclpattern{#3}

  \end{tikzpicture}}
}

\newcommandx{\descriptionpatt}[9][8={},9={}]
{
	\raisebox{0.6ex}{
  \begin{tikzpicture}[scale=0.35, baseline=(current bounding box.center), #1]
  	\useasboundingbox (0.0,-0.1) rectangle (#2+1.4,#3+1.1);
		
		\shadetherects{#7}

		\foreach \pos/\type in {#4}
		{
			\ifthenelse{\equal{\type}{v}}
			{
				\drawverticallines{#3}{\pos}{0.8pt}
			}
			{
				\drawhorizontallines{#2}{\pos}{0.8pt}
			}
		}

		\foreach \pos/\type in {#5}
		{
			\ifthenelse{\equal{\type}{v}}
			{
				\drawverticallines{#3}{\pos}{0.4pt}
			}
			{
				\ifthenelse{\equal{\type}{d}}
				{
					\draw[densely dashed] (\pos,0) -- (\pos,#3+1);
				}
				{
					\drawhorizontallines{#2}{\pos}{0.4pt}
				}
			}
		}
		
		\drawascendingrestrictions{#8}
		\drawspecialbox{#9}
		
		\foreach \x/\y/\type in {#6}
		{
			\ifthenelse{\equal{\type}{a}}
			{
				\draw (\x,\y) circle (6pt);
				\filldraw (\x,\y) circle (3pt);
			}
			{
				\filldraw (\x,\y) circle (5pt);
			}
		}
    
  \end{tikzpicture}}
} 

\newcommand{\descpatte}[8]
{
	\begin{tikzpicture}[scale=0.35, baseline=(current bounding box.center), #1]
		\shadetherects{#5}
		\draw (0.01,0.01) grid (#2+1-0.01,#3+1-0.01);
		
		\drawsolidshadedbox{#7}
		\drawspecialbox{#6}
		\drawspecialbox{#8}

		\foreach \x/\y/\type in {#4}
		{
			\ifthenelse{\equal{\type}{s}}
			{
				\draw (\x,\y) circle (6pt);
				\filldraw (\x,\y) circle (3pt);
			}
			{
				\filldraw (\x,\y) circle (5pt);
			}
		}
		
  \end{tikzpicture}
} 
\pgfmathsetmacro{\patttablecale}{1.05}
\pgfmathsetmacro{\pattdispscale}{0.80}
\pgfmathsetmacro{\patttextscale}{0.6}

\usepackage{hyperref}

\newcommand{\mmpatterna}[5]{
  \raisebox{0.6ex}{
  \begin{tikzpicture}[scale=0.35, baseline=(current bounding box.center), #1]
  	\useasboundingbox (0.0,-0.1) rectangle (#2+1.4,#2+1.1);
	
    \shadetheboxes{#4}
    
    \drawthegrid{#2}
    
    \fill[color = white!100, opacity=1, rounded corners = 1.5pt] (2,3.9) -- (1.1,3.9) -- (1.1,2.1) -- (2.1,2.1) -- (2.1,1.1) -- (3.9,1.1) -- (3.9,2.9) -- (2.9,2.9) -- (2.9,3.9) -- (2,3.9);
    \draw[color = black, rounded corners = 1.5pt] (2,3.9) -- (1.1,3.9) -- (1.1,2.1) -- (2.1,2.1) -- (2.1,1.1) -- (3.9,1.1) -- (3.9,2.9) -- (2.9,2.9) -- (2.9,3.9) -- (2,3.9);
    
    \fill[black] (2.5,2.5) node {$\scriptstyle1$};
    
    \drawspecialbox{#5}
    
    \drawtheclpattern{#3}
    
  \end{tikzpicture}}
}

\newcommand{\mmpatternb}[5]{
  \raisebox{0.6ex}{
  \begin{tikzpicture}[scale=0.35, baseline=(current bounding box.center), #1]
  	\useasboundingbox (0.0,-0.1) rectangle (#2+1.4,#2+1.1);
	
    \shadetheboxes{#4}
    
    \drawthegrid{#2}
       
    \fill[color = white!100, opacity=1, rounded corners = 1.5pt] (4,4.9) -- (3.1,4.9) -- (3.1,4.1) -- (4.1,4.1) -- (4.1,3.1) -- (4.9,3.1) -- (4.9,4.9) -- (4,4.9);
    \draw[color = black, rounded corners = 1.5pt] (4,4.9) -- (3.1,4.9) -- (3.1,4.1) -- (4.1,4.1) -- (4.1,3.1) -- (4.9,3.1) -- (4.9,4.9) -- (4,4.9);
    
    \fill[color = white!100, opacity=1, rounded corners = 1.5pt] (0.1,1) -- (0.1,1.9) -- (0.9,1.9) -- (0.9,0.9) -- (1.9,0.9) -- (1.9,0.1) -- (0.1,0.1) -- (0.1,1);
    \draw[color = black, rounded corners = 1.5pt] (0.1,1) -- (0.1,1.9) -- (0.9,1.9) -- (0.9,0.9) -- (1.9,0.9) -- (1.9,0.1) -- (0.1,0.1) -- (0.1,1);
        
    \drawspecialbox{#5}
    
    \drawtheclpattern{#3}
    
  \end{tikzpicture}}
}

\newcommand{\gen}{\mathsf{Gen}}
\newcommand{\mine}{\mathsf{Mine}}
\newcommand{\bisc}{\mathsf{BiSC}}

\newcommand{\Av}{\mathrm{Av}}
\newcommand{\fl}{\mathrm{fl}}
\newcommand{\sh}{\mathrm{sh}}
\newcommand{\forb}{\mathrm{forb}}
\newcommand{\subwords}{\mathrm{subwords}}

\newtheorem{theorem}{Theorem}[section]

\newtheorem{conjecture}{Conjecture}[theorem]
\newtheorem{lemma}[theorem]{Lemma}

\title{BiSC: An algorithm for discovering generalized permutation patterns}
\author{Henning Ulfarsson\thanks{Supported by grant no.\ 090038013-4 from the Icelandic Research Fund.}\\
Reykjavik University\\
\href{mailto:henningu@ru.is}{henningu@ru.is}
}
\date{\today}

\begin{document}
\maketitle

\begin{abstract}
  Theorems relating permutations with objects in other fields of mathematics are
  often stated in terms of avoided patterns. Examples include various classes of
  Schubert varieties from algebraic geometry (Billey and Abe 2013), commuting
  functions in analysis (Baxter 1964), beta-shifts in dynamical systems (Elizalde
  2011) and homology of representations (Sundaram 1994). We present a new
  algorithm, BiSC, that, given any set of permutations, outputs a conjecture
  for describing the set in terms of avoided patterns. The algorithm automatically conjectures the statements of known theorems
  such as the descriptions of smooth (Lakshmibai and Sandhya 1990) and
  forest-like permutations (Bousquet-M{\'e}lou and Butler 2007), Baxter
  permutations (Chung et al.\ 1978), stack-sortable (Knuth 1975) and
  West-2-stack-sortable permutations (West 1990). The algorithm has also been
  used to discover new theorems and conjectures related to the dihedral and alternating subgroups
  of the symmetric group, Young tableaux, Wilf-equivalences, and sorting devices.

\end{abstract}

\section{Introduction}
\label{sec:in}

Permutation patterns have established connections between certain subsets of
permutations and objects in other fields, including algebraic
geometry (Billey and Abe~\cite{BH}), analysis (Chung, Graham, Hoggatt and Kleiman~\cite{Baxter}, Graham~\cite{BaxterOLD}), dynamical
systems (Elizalde~\cite{beta}), algorithm analysis in computer science (Knuth~\cite{K}) and homology (Sundaram~\cite{Sun}).
Table~\ref{tab:props} gives examples of theorems of this nature.

We introduce the $\bisc$ algorithm, which was inspired by
an open problem posed by Billey~\cite{SaraTalk}, who asked whether a computer
could ``learn'' mesh patterns\footnote{In the name of the algorithm ``Bi'' is short for Billey, ``S'' is short for
Steingrimsson, the last name of Einar Steingrimsson, my postdoctoral advisor, and ``C'' is short for Claesson, the last name
of Ander Claesson, a colleague who has inspired me throughout my career.}. The algorithm can automatically conjecture all
of the known theorems in Table~\ref{tab:props}, and the new statements in
Table~\ref{tab:newprops} were all conjectured first by $\bisc$.

\begin{table}[htp]
  \begin{center}
    \begin{tabular}{lll}
      \toprule
      Property of perms.\         & Forbidden patterns                     & Reference               \\
      \midrule
      smooth                      & $1324$, $2143$                         & \cite[Thm.~1]{LS}       \\
      forest-like                 & $1324$, $(2143,\{(2,2)\})$             & \cite[Thm.~1]{BMB}      \\
      Baxter                      & $(2413,\{(2,2)\})$, $(3142,\{(2,2)\})$ & \cite[p.~383]{Baxter}   \\
      stack-sortable              & $231$                                  & \cite[Section 2.2.1]{K} \\
      West-$2$-stack-sortable     & $2341$, $(3241,\{(1,4)\})$             & \cite[Thm.~4.2.18]{W90} \\
      simsun                      & $(321,\{(1,0),(1,1),(2,2)\})$          & \cite[p.~7]{BC}         \\
      \bottomrule
    \end{tabular}
  \end{center}
  \caption{Statements of known theorems $\bisc$ can rediscover}
  \label{tab:props}
\end{table}%

\paragraph{Basic definitions.}
A \emph{permutation} of length $n$ is a bijection from the set $\{1, \dotsc, n\}$ to
itself. We write permutations in \emph{one-line notation}, where $\pi_1
  \pi_2 \dotsm \pi_n$ is the permutation that sends $i$ to $\pi_i$. The \emph{empty}
permutation is the unique permutation of length $0$, and is denoted $\epsilon$.
If $w = w_1 w_2 \dotsm w_k$ is a word of distinct integers, $\fl(w)$ is the
permutation obtained by replacing the $i$th smallest letter in $w$ with $i$.
This is called the \emph{flattening} of $w$. A permutation $\pi$ of length $n$
\emph{contains} a permutation $p$ of length $k$ if there exist indices $1 \leq
  j_1 < j_2 < \dotsm < j_k \leq n$ such that $\fl(\pi_{j_1} \pi_{j_2} \dotsm
  \pi_{j_k}) = p$. In this context $p$ is called a \emph{(classical) pattern} in
$\pi$ and the values $\pi_{j_1} \pi_{j_2} \dotsm \pi_{j_k}$ are an
\emph{occurrence} of $p$ in $\pi$.

Consider, for example, the permutation $\pi = 47318265$. This permutation contains
the pattern $231$ since the subword $482$ is an occurrence of $231$.

If $\pi$ does not contain $p$ then it
\emph{avoids} $p$. The permutation $\pi = 47318265$ from above avoids the pattern
$4321$.
We let $S$ denote the set of all permutations and $\Av(P)$
denote the permutations that avoid all the patterns in a set $P$. If $A$
is any set of permutations let $A_{\leq n}$ be the subset of permutations of
length at most $n$, and $A_n$ be the subset of permutations with length exactly $n$.

\begin{table}[htb]
  \begin{center}
    \begin{tabular}{lll}
      \toprule
      Property of perms.\                      & Forbidden patterns                & Reference                    \\
      \midrule
      dihedral subgroup                        & $16$ patterns                     & Conj.~\ref{conj:dihedral}    \\
      alternating subgroup                     & infinitely many patterns          & Conj.~\ref{conj:alternating} \\
      Young tableaux avoiding {\tiny\yng(3,2)} & four patterns                     & Conj.~\ref{conj:tableaux32}  \\
      Wilf-equival. for $\Av(231,p)$           & Different for each $p$            & \cite{LaraTalk}              \\
      $1$-quick-sortable                       & $321$, $2413$, $(2143,\{(2,2)\})$ & \cite{H}                     \\
      $1324$-restricted stack-sorting          & four patterns                     & Conj.~\ref{conj:stack}       \\
      \bottomrule
    \end{tabular}
  \end{center}
  \caption{Statements of new theorems and conjectures motivated by $\bisc$}
  \label{tab:newprops}
\end{table}%

\noindent
The algorithms treated here have been implemented in the Python library Permuta~\cite{permuta}.
The library contains a tutorial on how to use the algorithms.

\section{Learning patterns}
\label{sec:grim}

Given a set of permutations $A$, we call a set of patterns $P$ a \emph{basis}
for $A$ if $A = \Av(P)$, and no subset of $P$ has this property. The set $A$ is
called a \emph{permutation class} if $P$ only contains classical patterns.
There is a well-known algorithm whose input is a permutation class (known to have a
basis with patterns of length at most $m$) presented
as a list of permutations in length order that can find the basis. Assume we
input the permutations below.
\begin{align*}
   & 1,
  12,
  21,
  123,
  132,
  213,
  312,
  321,
  1234,
  1243,
  1324,
  1423,
  1432,
  2134,
  2143,
  3124,    \\
   & 3214,
  4123,
  4132,
  4213,
  4321.
\end{align*}
The first missing permutation is $231$, so we add it to our potential basis $P = \{231\}$.
When we reach the permutations of length $4$ we see that $1342$, $2314$,
$2341$, $2413$, $2431$, $3142$, $3241$, $3412$, $3421$, $4231$ and $4312$ are all missing.
All of these, except the last one, contain the pattern $231$, which is already in the basis, so we should expect them to be missing.
The last one, the permutation $4312$, does not contain the pattern $231$, so we extend the basis to
$P = \{231, 4312\}$. If the input had permutations of length $5$, we would continue in the
same manner: checking whether the missing permutations contain a previously forbidden pattern, and
if not, extending the basis by adding new classical patterns. This process continues until we
reach $m$, the length of the longest basis elements.

Classical patterns are sufficient to describe some sets of permutations. However, in other
cases, they do not suffice. Consider for example the West-$2$-stack-sortable permutations~\cite{W90}:
\begin{align*}
   & 1,
  12,
  21,
  123,
  132,
  213,
  231,
  312,
  321,
  1234,
  1243,
  1324,
  1342,
  1423,
  1432,
  2134,    \\
   & 2143,
  2314,
  2413,
  2431,
  3124,
  3142,
  3214,
  3412,
  3421,
  4123,
  4132,
  4213,
  4231,    \\
   & 4312,
  4321, \dotsc,
  35241, \dotsc.
\end{align*}
The first missing permutations are $2341$ and $3241$ so we would initially suspect
$\{2341,3241\} \subseteq P$. There are many missing permutations of length $5$, and all are consequences of $2341$ being forbidden, e.g., $34152$ is not in the
input since the subword $3452$ is an occurrence of $2341$. But one of the input
permutations is $35241$, which contains the previously forbidden pattern
$3241$. It seems that the presence of the $5$ inside the occurrence of $3241$
in $35241$ is important. This inspired West~\cite{W90} to define barred patterns,
but we will use the more expressive mesh patterns of Br\"and\'en and Claesson~\cite{BC},
which allow us to forbid letters from occupying certain regions in a pattern.
In the language of mesh patterns, we must find a shading $R$ such that the mesh pattern $(3241, R)$ is contained
in the permutation $3241$ but avoided by $35241$. In this case, it is easy to
guess $R = \{(1,4)\}$ (the top-most square between $3$ and $2$), which in fact
is the correct choice: the basis $P = \{2341, (3241,\{(1,4)\})\}$ exactly
describes the West-$2$-stack-sortable permutations, as was first shown by West~\cite{W90},
although his notation was in terms of the previously mentioned barred patterns.

Now consider the following input,
\begin{equation} \label{eq:difficult2}
  1,21,321,2341,4123,4321.
\end{equation}
When we see that the permutation $12$ is missing, we add the mesh pattern $(12,\emptyset)$ to our basis. The only
permutation of length $3$ in the input is $321$, which is consistent with our current
basis. We see that every permutation of length four is missing until the permutation $2341$, which
makes sense since each contains $(12,\emptyset)$.
But $2341$ also contains $12$; in fact it contains several occurrences: $23$, $24$ and $34$. These occurrences tell us that the following mesh patterns are actually allowed.
\begin{equation} \label{eq:length2allowed}
  \mpattern{scale=\pattdispscale}{ 2 }{ 1/1, 2/2 }{0/0,0/1,0/2,1/0,1/1,1/2,2/1} \qquad
  \mpattern{scale=\pattdispscale}{ 2 }{ 1/1, 2/2 }{0/0,0/1,0/2,1/0,1/2,2/1,2/2} \qquad
  \mpattern{scale=\pattdispscale}{ 2 }{ 1/1, 2/2 }{0/1,0/2,1/0,1/1,1/2,2/1,2/2}
\end{equation}
It is tempting to guess that what is forbidden is $\mpattern{scale=\patttextscale}{2}{ 1/1, 2/2 }{2/0}$. We
must revisit all the permutations not in the input to check whether they
contain our modified forbidden pattern. This fails for the permutation $231$. We must
add something to our basis, but what? Based on the mesh patterns in \eqref{eq:length2allowed}
we could add $\mpattern{scale=\patttextscale}{2}{ 1/1, 2/2 }{0/0,1/1,2/2}$, which is contained in $231$ but not
contained in any permutation in the input we have looked at up to now. From this small example, it
should be clear that it quickly becomes difficult to go back and forth in the input, modifying the currently
forbidden patterns and ensuring that they are not contained in any permutation in the input while still
being contained in the permutations not in the input. We need a more unified approach. But
first, if the reader is curious, the permutations in \eqref{eq:difficult2} are the permutations of length at most $4$ in the set
\begin{equation}
  \label{eq:av12}
  \Av \left( \mpattern{scale=\pattdispscale}{ 2 }{ 1/1, 2/2 }{0/0,1/1,2/2},\mpattern{scale=\pattdispscale}{ 2 }{ 1/1, 2/2 }{0/2,1/1,2/0} \right).
\end{equation}

The following section presents an algorithmic solution to this problem.

\subsection{The $\bisc$ algorithm}
In the example with West-$2$-stack-sortable permutations, we saw the mesh pattern
$(3241,\{(1,4)\})$ was added to the basis by considering a permutation in the input of length $5$. This
motivates the first step in the high-level description of $\bisc$:
\begin{enumerate}[(1)]
  \item \label{grimsketch1}$\mine$: Record mesh patterns that are contained in permutations in the input
\end{enumerate}
In the second step:
\begin{enumerate}[(1)]
  \addtocounter{enumi}{1}
  \item \label{grimsketch2}$\gen$: Infer the forbidden patterns from the allowed patterns found in step~\eqref{grimsketch1}
\end{enumerate}
More precisely, step~\eqref{grimsketch1} is accomplished with Algorithm~\ref{algo:stein} below.
\begin{algorithm}[htp]
  \DontPrintSemicolon
  \KwIn{$A$, a finite set of permutations; and $m$, an upper bound on the length of the patterns
    to search for}
  \KwOut{A list, $S$, consisting of tuples $(p,\sh_p)$ for all classical patterns, $p$, of length at most $m$. For each
    pattern $p$, $\sh_p$ contains the maximal shadings $R$ of $p$ such that $(p,R)$ has an occurrence in some
    input permutation}
  Initialize $S = \{ (p,\emptyset) \colon p \textrm{ classical pattern of length at most } m\}$\\
  \For {$\pi \in A$}
  {
    \For{$s \in \subwords_{\leq m}(\pi)$}
    { \nllabel{algo:steinsubw}
      Let $p = \fl(s)$\; \nllabel{algo:steinfl}
      Let $R$ be the maximal shading of $p$ for the occurrence $s$ in $\pi$\; \nllabel{algo:steinshmax}
      \If {$R \nsubseteq T$ for all shadings $T \in \sh_p$}
      {	\nllabel{algo:steinsh}
        Add $R$ to $\sh_p$ and remove subsets of $R$\;\nllabel{algo:steinshh}
      }
    }
  }
  \caption{$\mine$: The first step in $\bisc$ finds all allowed mesh patterns}\label{algo:stein}
\end{algorithm}
In line~\ref{algo:steinsubw}, $\subwords_{\leq m}(\pi)$ consists of subwords of length at most $m$ in the permutation $\pi$, so for example, $\subwords_{\leq 2}(1324) =
  \{13,12,14,32,34,24,1,3,2,4,\epsilon\}$.
In line~\ref{algo:steinshmax}, we let $R$ be the maximal shading that can be applied to the classical
pattern $p$ while ensuring that $s$ is still an occurrence of $(p,R)$, see Fig.~\ref{fig:mineEx}.

\begin{figure}[htp]
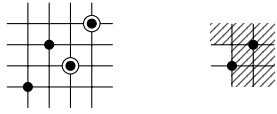

  \begin{center}
    \mpatternwwo{scale=\pattdispscale}{ 4 }{ 1/1, 2/3, 3/2, 4/4 }{}{3/2, 4/4}{} \qquad \mpattern{scale=\pattdispscale}{ 2 }{ 1/1, 2/2 }{0/2, 1/0, 1/1, 1/2, 2/0, 2/1, 2/2 }
    \caption{Applying the maximal shading to the classical pattern $12$ so that $s = 24$ is still an occurrence in $1324$}
    \label{fig:mineEx}
  \end{center}
\end{figure}
\noindent
In lines~\ref{algo:steinsh}-\ref{algo:steinshh}, $\sh_p$ is a set we maintain of maximal shadings of $p$. For example
if we see that
$
  \mpattern{scale=\patttextscale}{ 2 }{ 1/1, 2/2 }{0/0,1/1,2/2}
$
is an allowed shading, and then later come across
$
  \mpattern{scale=\patttextscale}{ 2 }{ 1/1, 2/2 }{0/0,1/1,2/2,0/2,2/0}
$
we will only store the second shading since the first has now become redundant.
We note that classical patterns $p$ that have no occurrence in any permutation $\pi$ in the input $A$
will appear in the output as $(p,\emptyset)$.

Theorems~\ref{thm:biscMainThm} and~\ref{thm:classical} below show that $\bisc$
will never produce patterns contained in the input and, in some cases, a correct description of the input permutations. These theorems rely on the
following two lemmas.

\begin{lemma} \label{lem:R'}
  If $(p,R)$ is any mesh pattern, of length at most $m$, that occurs in a permutation in $A$
  then there exists a mesh pattern $(p,R')$ in the output of $\mine(A,m)$ such that $R \subseteq R'$.
\end{lemma}

\begin{proof}
  Let $(p,R)$ be a mesh pattern of length at most $m$ that occurs in the permutation $\pi$
  in $A$. Let $w$ be the subword of $\pi$ which is an occurrence of $(p, R)$.
  At some stage in the for-loop in Algorithm~\ref{algo:stein}, we consider the permutation $\pi$ and the
  subword $w$. Either $R$ is maximal or is a subset of the maximal shading, $R'$ of $s$ in $\pi$.
  Furthermore, either of these shadings will be added to $\sh_p$, \emph{unless} there is already
  a shading in $\sh_p$ that is a superset. As the for-loop continues, the maximal shading of $s$ will
  never be removed from $\sh_p$ unless it is replaced with a superset, which proves the lemma.
\end{proof}

\begin{lemma} \label{lemma:steincl}
  If a set of permutations $A$ is defined by the avoidance of a (possibly infinite)
  list of classical patterns $P$ and $(q,\sh_q)$ with $\sh_q \neq \emptyset$ is in the output of $\mine(A_{\leq n},m)$
  then $q$ is not in the list $P$ and $\sh_q$ only contains the full shading of $q$.
\end{lemma}

\begin{proof}
  If $\sh_q$ does not contain the complete shading then $q$ is not in $A$, but appears as a classical
  pattern in some larger permutation $\pi$ in $A$. Since $q$ is not in $A$ it must contain one
  of the forbidden patterns $p$ in $P$, which would imply that $p$ also occurs in $\pi$, contradicting
  the fact that $\pi$ is in $A$.
\end{proof}

Step~\eqref{grimsketch2} in the high-level description of the algorithm is implemented
in Algorithm~\ref{algo:grim}, where we generate the forbidden patterns from the allowed
patterns found in step~\eqref{grimsketch1}.
\begin{algorithm}[htp]
  \DontPrintSemicolon
  \KwIn{$S$, a list of classical patterns along with shadings $\sh_p$, ordered by the length of the patterns}
  \KwOut{A list consisting of $(p,\forb_p)$ where $p$ is a classical pattern and $\forb_p$ is a set of minimal forbidden shadings}
  \For {$(p,\sh_p) \in S$}
  {
    Let $\forb_p$ be the minimal shadings of $p$ that are not contained in any member of $\sh_p$\; \nllabel{algo:grimforb}
    \For{$R \in \forb_p$}
    {
      \If{$R$ is a consequence of some shading in $\forb_q$ for a pattern $q$ contained in $p$}
      {\nllabel{algo:grimcons1}
        Remove $R$ from $\forb_p$\; \nllabel{algo:grimcons2}
      }
    }
  }
  \caption{$\gen$: The second step in $\bisc$ generates the forbidden patterns} \label{algo:grim}
\end{algorithm}
To explain lines~\ref{algo:grimcons1}--\ref{algo:grimcons2} in the algorithm assume $p = 1243$ and the shading, $R$, on the left in Fig.~\ref{fig:forbEx} is one of the shadings
in $\forb_p$ on line~\ref{algo:grimforb}.
Assume also that earlier in the outer-most for-loop, we generated $q = 12$ with the shading $R'$
on the right in Fig.~\ref{fig:forbEx}. Then the shading $R$ is removed from $\forb_p$ on line~\ref{algo:grimcons2}
since any permutation containing $(p,R)$, also contains $(q,R')$.
\begin{figure}[htp]
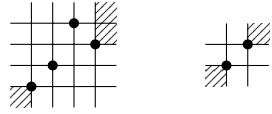

  \begin{center}
    \mpattern{scale=\pattdispscale}{ 4 }{ 1/1, 2/2, 3/4, 4/3 }{0/0, 4/3, 4/4 }\qquad
    \mpattern{scale=\pattdispscale}{ 2 }{ 1/1, 2/2 }{0/0, 2/2 }
    \caption{The mesh patterns $(p,R)$ and $(q,R')$}
    \label{fig:forbEx}
  \end{center}
\end{figure}

We now define the full algorithm as $\bisc(A,m) = \gen(\mine(A,m))$. The next two theorems
show that the output of $\bisc$ is correct, in the sense that it will never
output patterns contained in a permutation in the input, and can give
a proof of the description if a bound is known on the length of the patterns.

\begin{theorem} \label{thm:biscMainThm}
  Let $A$ be any set of permutations. Then for all positive integers $N \geq n$ and $m$
  \[
    A_{\leq n} \subseteq \Av(\bisc(A_{\leq N},m))_{\leq n}.
  \]
  Furthermore, if the set $A$ is defined in terms of a finite list of patterns (classical or not)
  whose longest pattern has length $k$, and if $N \geq n \geq k$, $m \geq k$, then there is equality
  between the two sets above.
\end{theorem}

\begin{proof}
  To prove the subset relation let $\pi$ be a permutation in $A_{\leq n}$ and $(p,R)$ be a mesh pattern of length at most $m$
  contained in $\pi$. By Lemma~\ref{lem:R'} there is a pattern $(p,R')$ in the output from
  $\mine(A_{\leq N},m)$ such that $R \subseteq R'$.
  This implies that $R$ is not one of the shadings in $\forb_p$ and therefore $(p,R)$
  is not in the output of $\bisc(A_{\leq n},m)$, and therefore $\pi$ is in $\Av(\bisc(A_{\leq N},m))_{\leq n}$.
  
  To prove the equality of the sets, let $\pi$ be a permutation that is not in the set on the left because it contains
  a mesh pattern $(p, R)$ from the list defining $A$. If $p$ never occurs as a classical pattern in a
  permutation in $A_{\leq N}$ then $(p,\emptyset)$ is in the output of $\mine(A_{\leq N},m)$ which implies
  that $(p,\emptyset)$ is in the output of $\bisc(A_{\leq N},m)$. This implies that $\pi$ is not in the set
  on the right. If, however, $p$ occurs in some permutations in $A_{\leq N}$, then every occurrence of a mesh
  pattern $(p, R')$ must satisfy $R \nsubseteq R'$. This implies that when $\forb_p$ is created in
  line~\ref{algo:grimforb} in Algorithm~\ref{algo:grim}, it contains $R$ or a non-empty subset, $R''$, of it.
  Without loss of generality, we assume that $R''$ is not removed due to redundancy in line~\ref{algo:grimcons2}.
  This implies that $(p, R'')$ is in the output of $\bisc(A_{\leq N},m)$, so $\pi$ is not in the set on the right.
\end{proof}

We can strengthen the previous theorem if the input $A$ to $\bisc$ is defined in terms of classical patterns.

\begin{theorem} \label{thm:classical}
  Let $A$ be a set of permutations defined by the avoidance of a (possibly infinite)
  list $P$ of classical patterns. Then, the output of $\bisc(A_{\leq n},m)$ for any $n,m$ will consist only
  of classical patterns.
  If the longest patterns in the list $P$ have length $k$, then
  $A = \Av(\bisc(A_{\leq n},m))$ for any $n,m \geq k$.
\end{theorem}
\begin{proof}
  This follows from Lemma~\ref{lemma:steincl}.
\end{proof}

Some properties require us to look at very large permutations to discover patterns. For
example, one must look at permutations of length $8$ in the set
\begin{equation*}
  \Av \left( \mpattern{scale=\pattdispscale}{ 2 }{ 1/1, 2/2 }{0/0,1/1,2/2},\mpattern{scale=\pattdispscale}{ 2 }{ 1/1, 2/2 }{0/2,1/1,2/0} \right) = \epsilon, 1,21,321,2341,4123,4321, \dotsc.
\end{equation*}
we considered above
to see that the mesh pattern $\mpattern{scale=\patttextscale}{ 2 }{ 1/1, 2/2 }{0/0,0/2,1/0,2/0,2/1,2/2}$ is allowed.
But as examples in the next section show, it often suffices to look at permutations of length $m+1$ when
searching for mesh patterns of length $m$.

For some input sets $A$, there is redundancy in the output $\bisc(A,m)$, in the sense that
some patterns can be omitted without changing $\Av(\bisc(A,m))$. The implementation of the
algorithm has a pruning step to remove redundancies like this.
This is never a problem in the following applications, so we will omit further details.

\subsection{Applications of $\bisc$}
\label{subsec:appl}

The $\bisc$ algorithm can, as was noted in the introduction, rediscover several known
theorems, some of which we listed in Table~\ref{tab:props}. This section presents new results and conjectures
first discovered by $\bisc$. At least one of them, Conjecture~\ref{conj:tableaux32}, would perhaps have
been nearly impossible to discover without some automation due to how complicated
the patterns in the statement are.

\paragraph{The dihedral and alternating subgroups.}
One of the most well-known subgroups of the symmetric group is the dihedral group $D_n$.
Let $D$ denote the union of $D_n$ for all $n$.
If run $\bisc$ on the union of the dihedral subgroups of permutations of length at most $4$
and search for patterns of length $4$, i.e., compute $\bisc(D_{\leq 4},4)$, the output
is a list of $16$ classical patterns given in the following conjecture.
\begin{conjecture} \label{conj:dihedral}
  The union of the dihedral subgroups consists of permutations avoiding the classical
  patterns $3142$, $1342$, $2314$, $2134$, $3241$, $4231$, $3421$, $4213$, $2413$, $1324$
  $3124$, $1423$, $2431$, $4312$, $4132$, $1243$.
\end{conjecture}
Atkinson and Beals~\cite{subgrp} classified which subgroups of the symmetric group have descriptions
in terms of classical patterns but do not give the basis for the dihedral group.

Another well-known class of subgroups are the alternating subgroups. Let $A$ be
the union of all these subgroups. Running $\bisc(A_{\leq n},m)$ for any $n \geq m$
always outputs fully shaded mesh patterns. This seems to be the only way to describe
these subgroups with mesh patterns:
\begin{conjecture} \label{conj:alternating}
  Let $A$ be the union of the alternating subgroups, and let $P$ be a basis of
  patterns such that $A = \Av(P)$. Then
  \[
    P = \{ (p, \{0,\dotsc,n\} \times \{0,\dotsc,n\}) \colon p \in S_n \setminus A \text{ for some } n \}.
  \]
\end{conjecture}

\paragraph{Forbidden shapes in Young tableaux.}
There is a well-known bijection, the Robinson-Schensted-Knuth-correspondence (RSK) between permutations of length $n$
and pairs of Young tableaux of the same shape. Let $A$ be the set of permutations whose image tableaux are hook-shaped. The output of $\bisc(A_{\leq 5},4)$ is the following four patterns.
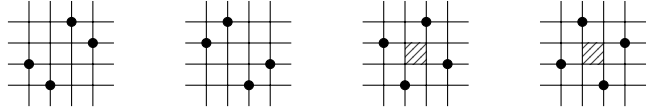
\[
  \mpattern{scale=\pattdispscale}{ 4 }{ 1/2, 2/1, 3/4, 4/3 }{} \qquad
  \mpattern{scale=\pattdispscale}{ 4 }{ 1/3, 2/4, 3/1, 4/2 }{} \qquad
  \mpattern{scale=\pattdispscale}{ 4 }{ 1/3, 2/1, 3/4, 4/2 }{2/2} \qquad
  \mpattern{scale=\pattdispscale}{ 4 }{ 1/2, 2/4, 3/1, 4/3 }{2/2}
\]
Here, we have rediscovered an observation made by Atkinson~\cite[Proof of Lemma 9]{A}.
See also Albert~\cite{YoungClasses} for connections between classical patterns
and Young tableaux.
We can view the property of being a hook-shaped tableau as not containing
the shape {\tiny\yng(2,2)}. It is natural to ask if we can, in general, describe the permutations corresponding to tableaux not containing a general
shape $\lambda$. E.g., let $A$ be the permutations whose tableaux avoid {\tiny\yng(3,2)}.
The output of $\bisc(A_{\leq 6},5)$ is $25$ mesh patterns that can be simplified to give the
following.
\begin{conjecture} \label{conj:tableaux32}
  The permutations whose tableaux under the RSK-map avoid the shape {\tiny\yng(3,2)}
  are precisely the avoiders of the marked mesh patterns (\cite[Definition 4.5]{Unific}) below.
  \[
    \mmpatterna{scale=\pattdispscale}{ 4 }{ 1/2, 2/1, 3/4, 4/3 }{}{0/0/1/1/{},4/4/5/5/{}} \qquad
    \mmpatternb{scale=\pattdispscale}{ 4 }{ 1/3, 2/4, 3/1, 4/2 }{}{1/3/2/4/1,3/1/4/2/{}} \qquad
    \mmpattern{scale=\pattdispscale}{ 4 }{ 1/3, 2/1, 3/4, 4/2 }{2/2}{0/0/2/1/1,3/4/5/5/{}} \qquad
    \mmpattern{scale=\pattdispscale}{ 4 }{ 1/2, 2/4, 3/1, 4/3 }{2/2}{0/0/1/2/1,4/3/5/5/{}}
  \]
\end{conjecture}
The meaning of a marked region is that it should contain at least one point, and therefore
the first pattern above corresponds to $9$ classical patterns.
Crites et al.~\cite{CPW} showed that if a permutation $\pi$ contains a separable classical pattern
$p$, then the tableaux shape of $p$ is contained in the tableaux shape of $\pi$. The observation of Atkinson and Beals
and the proposition above suggest a more general result if one considers mesh
patterns instead of classical patterns.

\paragraph{Forbidden tree patterns.}
Pudwell et al.~\cite{LaraTalk} used the bijection between binary trees and $231$-avoiding permutations and
the $\bisc$ algorithm to discover a correspondence between tree patterns and
mesh patterns. This can be used to generate Catalan-many Wilf-equivalent sets of the form $\Av(231,p)$
where $p$ is a mesh pattern. The tree in Fig.~\ref{fig:bintree} gives the Wilf-equivalence between the
sets $\Av(231,654321)$ and $\Av(231,(126345,\{(1,6),(4,5),(4,6)\}))$.

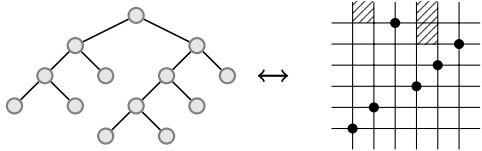
\begin{figure}[htb!]
  \begin{center}
    \begin{tikzpicture}[scale = 0.4, place/.style = {circle,draw = black!50,fill = gray!20,thick,inner sep=2pt}, auto]

      \node [place] (1) at (0,2)     {};
      \node [place] (2) at (2,1)	     {};
      \node [place] (3) at (1,0)     {};
      \node [place] (4) at (0,-1)	     {};
      \node [place] (5) at (-2,1)	     {};
      \node [place] (6) at (-1,0)     {};
      \node [place] (7) at (-3,0)     {};
      \node [place] (8) at (-4,-1)     {};
      \node [place] (9) at (-2,-1)     {};
      \node [place] (10) at (3,0)     {};
      \node [place] (11) at (2,-1)     {};
      \node [place] (12) at (1,-2)     {};
      \node [place] (13) at (-1,-2)     {};

      \draw [-,semithick] 			(1) to (2);
      \draw [-,semithick] 			(2) to (3);
      \draw [-,semithick] 			(3) to (4);
      \draw [-,semithick] 			(1) to (5);
      \draw [-,semithick] 			(5) to (6);
      \draw [-,semithick] 			(5) to (7);
      \draw [-,semithick] 			(7) to (8);
      \draw [-,semithick] 			(7) to (9);
      \draw [-,semithick] 			(2) to (10);
      \draw [-,semithick] 			(3) to (11);
      \draw [-,semithick] 			(4) to (12);
      \draw [-,semithick] 			(4) to (13);

      \draw [<->, thick] (4,0) to (5,0);

      \node [] (x) at (5.5,0) [label = right:$\mpattern{scale=\pattdispscale}{ 6 }{ 1/1, 2/2, 3/6, 4/3, 5/4, 6/5 }{1/6,4/5,4/6}$] {};

    \end{tikzpicture}
    \caption{A binary tree and its corresponding mesh pattern, see~\cite{LaraTalk} }
    \label{fig:bintree}
  \end{center}
\end{figure}

\paragraph{Sorting with restricted stacks}
Cerbai, Claesson, and Ferrari~\cite{cerbai_stack_2020} introduce a generalization of
stack-sorting with two stacks in series by forcing the (sub) permutation on the first
stack to avoid a pattern. This paper was followed by Baril, Cerbai, Khalil, and Vajnovszki~\cite{BARIL2021106138}
where they consider the avoidance of several patterns. $\bisc$ can be used to conjecture
several descriptions of the sets of permutations that are sortable under the restrictions
derived in these papers. It can also conjecture new results, such as the following.

\begin{conjecture} \label{conj:stack}
  The set of permutations that are sortable under the restriction that the first stack
  avoids $1324$ are the permutations avoiding
  \[
  \mpattern{scale=\pattdispscale}{ 3 }{ 1/1, 2/3, 3/2 }{0/1, 0/2, 2/0} \qquad
  \mpattern{scale=\pattdispscale}{ 3 }{ 1/1, 2/3, 3/2 }{0/3, 1/2} \qquad
  \mpattern{scale=\pattdispscale}{ 5 }{ 1/4, 2/2, 3/3, 4/1, 5/5 }{} \qquad
  \mpattern{scale=\pattdispscale}{ 5 }{ 1/5, 2/2, 3/3, 4/1, 5/4 }{}
  \]
\end{conjecture}

\paragraph{$1$-quicksortable permutations.}
The following sorting method was proposed by Claesson~\cite{AC} as a single
pass of quicksort (See, e.g., Hoare~\cite{Hoare}): If the input permutation is empty, return the
empty permutation. If the input permutation $\pi$ is not empty and contains
strong fixed points\footnote{A letter in a permutation is a \emph{strong fixed
    point} if every letter preceding it is smaller, and every letter following it
  is larger.} let $x$ be the right-most such point, write $\pi = \alpha x \beta$
and recursively apply the method to $\alpha$ and $\beta$. If $\pi$ does not
contain a strong fixed point, we move every letter in $\pi$ smaller than
the first letter to the start of the permutation, thus creating a strong fixed point. Let $A$ be the set of
permutations that are sortable in at most one pass under this operator. The
output of $\bisc(A_{\leq 5},4)$ is the following.
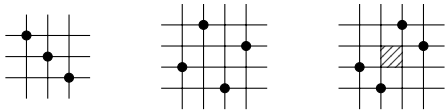
\[
  \mpattern{scale=\pattdispscale}{ 3 }{ 1/3, 2/2, 3/1 }{} \qquad
  \mpattern{scale=\pattdispscale}{ 4 }{ 1/2, 2/4, 3/1, 4/3 }{} \qquad
  \mpattern{scale=\pattdispscale}{ 4 }{ 1/2, 2/1, 3/4, 4/3 }{2/2}
\]
Stuart Hannah~\cite{H} verified that the sortable permutations are precisely the avoiders
of these three patterns and that they are in bijection
with words on the alphabet $ (a+(ab^\star)(cc^\star))^\star$, which are counted by the sequence~\cite[A034943]{OEIS}.

\section{Future work}

It is important to remember that unless certain conditions are met, as in Theorems~\ref{thm:biscMainThm} and~\ref{thm:classical},
the output of $\bisc$ is a conjecture. A natural next step is to develop a stronger algorithm
that can prove the conjectures output by $\bisc$.

In cases where the human mathematician (or a future algorithm) has verified the output
of $\bisc$, a natural follow-up question is to try to enumerate the permutations in the
set by length. There already is an algorithm that can do this for many sets of permutations
defined by classical pattern avoidance; see Albert, Bean, Claesson, Nadeau, Pantone, Ulfarsson~\cite{combinatorial-exploration}. It would be interesting to see if
this algorithm can be extended to mesh patterns.

\bibliographystyle{amsalpha}
\bibliography{magulf_refs}

\end{document}